\documentclass[10pt,a4paper]{article}

\usepackage[english]{babel}
\usepackage{graphicx,epstopdf,epsfig}
\usepackage{amsfonts,fancyhdr,graphics, hyperref,amsmath,amssymb}
\usepackage{verbatim}
\usepackage{url}
\usepackage{amsthm}
\usepackage{hyperref}
\usepackage{titlesec}

\newcommand{\Names}{David Hartman, Milan Hlad\'{i}k}
\newcommand{\Title}{{\bf Regularity radius: Properties, approximation and a not a priori exponential algorithm}}

\newcommand{\tluste}[1]{\mbox{\mathversion{bold}$ #1 $}}
\newcommand{\vr}[1]{{{#1}}}

\newcommand{\mace}[1]{{{#1}}}
\newcommand{\mna}[1]{{\mathcal{#1}}}

\newcommand{\imace}[1]{\mbox{$\tluste{#1}$}}

\def\Rad#1{{#1}_\Delta}		
\newcommand{\R}[0]{{\mathbb{R}}}

\def\diag{{\mbox{diag}}}
\def\sgn{{\rm{sgn}}}
\def\eps{{\varepsilon}}

\def\nref#1{$(\ref{#1})$}
\newcommand{\seznam}[1]{{\{1, \ldots, {#1}\}}}

\DeclareMathOperator{\rr}{r}
\def\quo#1{``{#1}''}  

\newcommand{\maxim}[2]{{\max {\{#1; \ {} #2 \}} }}
\newcommand{\mmid}[0]{;\,}              

\newtheorem{theorem}{Theorem}
\numberwithin{theorem}{section}

\newtheorem{proposition}[theorem]{Proposition}
\newtheorem{lemma}[theorem]{Lemma}
\newtheorem{example}[theorem]{Example}
\newtheorem{question}[theorem]{Question}

\titleformat{\section}[runin]
  {\normalfont\large\bfseries}{\thesection}{1em}{}
\titleformat{\subsection}[runin]
  {\normalfont\normalsize\bfseries}{\thesubsection}{1em}{}
\titleformat{\abstract}[runin]
  {\normalfont\normalsize\bfseries}{\thesubsection}{1em}{}

\makeatletter
\if@titlepage
   \renewenvironment{abstract}{%
       \titlepage
       \null\vfil
       \@beginparpenalty\@lowpenalty
         \paragraph{\abstractname:}
         \@endparpenalty\@M
       }%
      {\par\vfil\null\endtitlepage}
\else
   \renewenvironment{abstract}{%
       \if@twocolumn
         \section*{\abstractname}%
       \else
         \small
         \paragraph{\abstractname:}
       \fi}
       {\if@twocolumn\else\par\bigskip\fi}
\fi
\makeatother

\begin{document}

\bibliographystyle{plain}

\setcounter{page}{1}

\thispagestyle{empty}

 \title{\Title}
 
 \date{}

\author{
David Hartman\thanks{Charles University, Faculty  of  Mathematics  and  Physics, Department of Applied Mathematics, Malostransk\'e n\'{a}m.~25, 11800, Prague, Czech Republic and Institute of Computer Science, Czech Academy of Sciences, Prague, Czech Republic (hartman@cs.cas.cz).}
\and
Milan Hlad\'{i}k\thanks{Charles University, Faculty  of  Mathematics  and  Physics, Department of Applied Mathematics, Malostransk\'e n\'{a}m.~25, 11800, Prague, Czech Republic (milan.hladik@matfyz.cz).\newline {\bf Funding}: This work was funded by the Czech Science Foundation Grant P403-18-04735S.}
}

\markboth{\Names}{\Title}

\maketitle

\vspace{-1cm}
\begin{abstract}
The radius of regularity sometimes spelled as the radius of nonsingularity is a measure providing the distance of a given matrix to the nearest singular one. Despite its possible application strength this measure is still far from being handled in an efficient way also due to findings of Poljak and Rohn providing proof that checking this property is NP-hard for a general matrix. There are basically two approaches to handle this situation. Firstly, approximation algorithms are applied and secondly, tighter bounds for radius of regularity are applied. Improvements of both approaches have been recently shown by Hartman and Hlad\'{i}k (doi:10.1007/978-3-319-31769-4\_9) utilizing relaxation of computation to semidefinite programming. An estimation of the regularity radius using either of mentioned ways is usually applied to general matrices considering none or just weak assumptions about the original matrix. Surprisingly less explored area is represented by utilization of properties of special class of matrices as well as utilization of classical algorithms extended to be used to compute the considered radius. This work explores a process of regularity radius analysis and identifies useful properties enabling easier estimation of the corresponding radius values. At first, checking finiteness of this characteristic is shown to be a polynomial problem along with determining a maximal bound on number of nonzero elements of the matrix to obtain infinite radius. Further, relationship between maximum (Chebyshev) norm and spectral norm is used to construct new bounds for the radius of regularity. Considering situations where known bounds are not enough, a new method based on Jansson-Rohn algorithm for testing regularity of an interval matrix is presented which is not a priory exponential along with numerical experiments. For a situation where an input matrix has a special form, several corresponding results are provided such as exact formulas for several special classes of matrices, e.g., for totally positive and inverse non-negative, or approximation algorithms, e.g., rank-one radius matrices. For tridiagonal matrices, an algorithm by Bar-On, Codenotti and Leoncini is utilized to design a polynomial algorithm to compute the radius of regularity.
\end{abstract}


\subsection*{AMS}
15A60, 65G40.


\section{Introduction} \label{intro-sec}

Nonsingularity of a matrix is well known to be an important property in many applications of linear algebra. Let us mention an important area of systems stability studied for example for linear time-invariant dynamical systems with parameters having uncertain values~\cite{RavanbodEtAl2017}. Within these systems questions like distance to instability or minimum stability degrees are solved. This means that more than an actual regularity of a matrix we are interested in the distance to the nearest singular matrix. To express this problem in more realistic terms let $A$ be a matrix containing results of measurements and subsequent computations. Even putting numerical problems aside there can be uncertainty $\epsilon_{i,j}$ in each element $a_{i,j}$ of the matrix $A$. To be able to efficiently account for variation of the original matrix determining a distance to a singular one it is more suitable to introduce fewer parameters. Let us start with a simple case introducing just one parameter $\delta$ and considering any matrices having their values componentwisely between $A - \delta ee^T$ and $A + \delta ee^T$, where $e$ is the vector of ones. In the language of interval linear algebra~\cite{Moore2009}, we can consider the following interval matrix $\imace{A}_{\delta}:=[A-\delta ee^T,A+\delta ee^T]$. 
An interval matrix $\imace{A}$ is called \emph{regular} is it consists merely of nonsingular matrices; otherwise, it is called \emph{singular} (or \emph{irregular}). The search for the distance to a singular matrix can be rephrased as a search for the minimal $\delta$ such that the interval matrix $\imace{A}_{\delta}$ becomes singular. 

Considerations in the above paragraph give rise to the following definition. Let $A \in \R^{n \times n}$ be a square real value matrix. Then {\em the regularity radius}, sometimes called the radius of nonsingularity, can be defined using the following definition
\begin{align}\label{eq:regrad}
\rr(A)
=\min\{\delta\geq0\mid \imace{A}_{\delta} \mbox{ is singular}\},
\end{align}
where $\imace{A}_{\delta}:=[A-\delta ee^T,A+\delta ee^T]$. This definition allows us to produce matrices via elementwise perturbation of the original elements. An interval of such perturbation is the same for all elements. This approach can be too restrictive considering possible application of the regularity radius. For these reasons, a generalized version of the above definition is often considered. Let $\Delta \in \R^{n \times n}$ be a square real value non-negative matrix. Then a generalized version of the regularity radius is
\begin{align}\label{eq:genregrad}
\rr(A,\Delta)
=\min\{\delta\geq0\mid \imace{A}_{\delta\Delta} \mbox{ is singular}\},
\end{align}
where $\imace{A}_{\delta\Delta}:=[A-\delta \Delta ,A+\delta \Delta]$. 

There exist several useful properties of regularity radius in literature that are helpful when considered matrix can be constructed using algebraic expression of, in some sense, simpler matrices. Let us mention these properties for sake of completeness. Let $A$, $B$ and $C$ be square real valued matrices and let $\Delta \in \R^{n \times n}$ be a square real valued non-negative matrix. The following assertions are true~\cite{Roh2012b}.
\begin{enumerate}
\item The radius of sum of matrices cannot exceed the sum of individual radii, i.e.,
\begin{equation*}
d(A + B, \Delta) \leq d(A, \Delta) + d(B, \Delta).
\end{equation*}
\item The larger radius matrix, the smaller regularity radius, i.e.,
\begin{equation*}
0 \leq \Delta \leq \Delta' \text{ implies } d(A, \Delta) \geq d(A, \Delta').
\end{equation*}
\item Multiplication by a constant does not significantly affect complexity of the computation, i.e.,
\begin{equation}\label{eq:rad_mult_constant}
d(\alpha A, \beta \Delta) = \frac{|\alpha|}{\beta} d(A,\Delta) \text{ for } \alpha \in \mathbb{R} \text{ and } \beta > 0.
\end{equation}
\end{enumerate}

Determining the regularity radius using directly one of its definitions is complicated even considering the simpler version. For a straight computation, Poljak and Rohn have shown an analytical formula~\cite{PolRoh1993} which reads as
\begin{align}\label{eqRdel}
\rr(A,\Delta)
=\frac{1}{\max_{y,z\in\{\pm1\}^n}\rho_0(A^{-1}D_y\Delta D_z)}.
\end{align}
where $D_y$ is a diagonal matrix having $y$ as its diagonal, i.e., $(D_y)_{ii} = y_i$ and $(D_y)_{ij} = 0$ for $i\not=j$, and $\rho_0$ is the real spectral radius providing maximum from absolute values of real eigenvalues of the matrix and equal to 0 if no such eigenvalue exists. This equation has been proven by Poljak and Rohn~\cite{PolRoh1993} using one of the equivalent conditions for regularity of an interval matrix~\cite{Roh1989positive}. Considering $\Delta$ to be a matrix $ee^T$, i.e., a matrix consisting of all ones, and substituting this $\Delta$ to the above defined formula results in the following simpler form~\cite{PolRoh1993}:
\begin{align*}
\rr(A,\Delta)
=\frac{1}{\| A^{-1}\|_{\infty,1}},
\end{align*}
where $\|\cdot\|_{\infty,1}$ is the matrix norm defined as
\begin{equation}\label{eq:rnorm}
\|\mace{M}\|_{\infty,1}
:=\maxim{\|\mace{M}\vr{x}\|_1}{\|\vr{x}\|_\infty=1}
=\maxim{\|\mace{M}\vr{z}\|_1}{\vr{z}\in\{\pm1\}^n}.
\end{equation}

For a general matrix $M$, a computation of the above defined norm has been shown to be an NP-hard problem~\cite{PolRoh1993}, and consequently also the problem of computing the regularity radius of $M^{-1}$ is NP-hard. This result has motivated two commonly studied approaches to handle computation of the regularity radius. The first approach is to develop various bounds for its values, while the second approach utilizes approximation algorithms. 

Considering the first approach, the corresponding bounds are mostly based on utilization of different norms of original matrix or variations of its spectral radius. One of the first results providing simple bounds has been proven by Demmel~\cite{Demmel1992} using Perron-Frobenius theorem and block Gaussian elimination, see these bounds in following equation
\begin{align*}
\frac{1}{\rho(A, \Delta)}
\leq 
\rr(A,\Delta)
\leq
\frac{1}{\max_{ij}(|A^{-1}_{ij}|\Delta_{ij})}.
\end{align*}

Following this work, Rump~\cite{Rum1994} has shown that for an interval matrix $\imace{A}_{\Delta} =[A-\Delta ,A+\Delta]$ with a central matrix $A$ having its singular values ordered as $\sigma_1(A) \geq \sigma_2(A) \geq \dots \geq \sigma_n(A)$ a condition $\sigma_n(A) > \sigma_1(\Delta)$ implies regularity of the original interval matrix $\imace{A}_{\Delta} = [A - \Delta, A + \Delta]$. He has also mentioned another criterion for regularity of interval matrix $\imace{A}_{\Delta}$ that reads as $\rho(|A^{-1}|\Delta) < 1$, where $\rho$ stands for the standard spectral radius of a matrix. Following these criteria several bounds can be produced. Already Rump in the mentioned work~\cite{Rum1994} has provided two lower bounds based on the above given conditions along with their mutual comparison, see the resulting bounds below:
\begin{align*}
\frac{1}{\rho(A, \Delta)}
\leq \rr(A,\Delta),
\qquad 
\frac{1}{\sigma_1(A)/\sigma_n(\Delta)}
\leq \rr(A,\Delta).
\end{align*}

Conditions about regularity of an interval matrix have been later extended by Rohn~\cite{Roh1996} who, besides the above defined lower bound, provided a new upper bound defined as follows:
\begin{align}\label{rUpperBoundRohn}
\rr(A,\Delta)
\leq \frac{1}{\max_i(|A^{-1}|\Delta)_{ii}}.
\end{align}

Let us note that this upper bound has also be reproven by Rump~\cite{Rum1997b}. In the same work, Rump has shown that for nonnegative invertible matrices, e.g., for M-matrices, the regularity radius is equal to the mentioned lower bound:
\begin{equation*}
\text{ If } A\in\R^{n\times n} \text{ is nonnegative invertible and } 0 \leq \Delta \in \R^{n\times n}, \text{ then }\quad \frac{1}{\rho(A, \Delta)}
= \rr(A,\Delta).
\end{equation*}

Motivated by this result he has also utilized this lower bound as a tool to tighten interval for the regularity radius. Tightening is realized via adopting the following equation determining parametrically the upper and the lower bounds
\begin{equation}\label{eq:rumpbound}
\frac{1}{\rho(A, \Delta)}
\leq 
\rr(A,\Delta)
\leq
\frac{\gamma(n)}{\rho(A, \Delta)}.
\end{equation}
He has shown that $\gamma(n) = 2.4 \cdot n^{1.7}$, which proves and extends the conjecture given by Demmel. He has also conjectured that $\gamma(n) = n$, supported by the property that no upper bound can be better than this function. Finally, it has been shown that this upper bound is sharp up to a constant factor as $n \leq \gamma(n) \leq (3 + 2\sqrt{2})n$.

As mentioned above, we can use different approach and utilize approximation algorithms to compute values of regularity radius. There is one method of Kolev~\cite{Kol2011}. He makes use of interval analysis namely for the search of real maximum magnitude eigenvalue of an associated generalized eigenvalue problem. His solution is an iterative process requiring several conditions on the whole system. Considering computational complexity, this method is not a priori exponential, but requires some sign invariancy of the corresponding eigenvalue problem. 

Another method is represented by the recent result of Hartman and Hlad\'{i}k~\cite{Hartman2016} that provides a randomized algorithm for computing $\rr(A)$ based on a semidefinite approximation of the corresponding matrix norm~\eqref{eq:rnorm}. As such a semidefinite relaxation can be solved with arbitrary a priori given precision $\epsilon$, see \cite{Grotschel1981}; this result provides the following bounds for matrix norm when considering an arbitrary matrix $M$:
\begin{equation}
	\alpha \gamma \leq \|M_{ij}\|_{\infty, 1} \leq \gamma + \epsilon
\end{equation}
where $\gamma$ is an optimal solution for a norm and $\alpha = 0.87856723$ is the Goemans-Williamson value characterizing the approximation ratio of their approximation algorithm for MaxCut~\cite{Goemans1995}.

There are some extensions of the basic notion of regularity radius. One example is represented by an extension of the regularity radius to structured perturbations pioneered by Kolev \cite{Kol2014b} and further studied in \cite{KolSka2017}. Another extension to a radius of (un)solvability of linear systems was presented by Hlad\'{\i}k and Rohn~\cite{HlaRoh2016a}.

\subsection{Structure of the paper}
The results presented in this paper follow the usual process of analyzing the regularity radius of a given matrix and provide thus suggestions to employ effective processing. There are two definitions of radius of regularity -- the simpler one, see equation~\eqref{eq:regrad}, and more general one, see equation~\eqref{eq:genregrad}. This work contributes to bounding or computing both types. While the paper is organized according to expected process of analyzing regularity radius results for both types are sometimes mixed. We always indicate which type of generality is considered on appropriate places of the paper.

One of the first questions when analyzing the regularity radius should be concerned with utilization of possibly special type of the corresponding matrix. The special type of matrices might be expected due to a frequent determination by a particular structure by the problem under study -- consider common occurrences of tridiagonal matrices in real world problems. This area is surprisingly less explored although some of the results are easy to achieve and strong at the same time. Section~\ref{s:specialmats} presents some observations concerning special types of matrices. The first subsection~\ref{ss:tridiagonalmat} presents utilization of the algorithm developed by by Bar-On, Codenotti and Leoncini~\cite{Baron1996} that represents a polynomial time algorithm to compute the regularity radius for tridiagonal matrices. What follows are results providing exact formulas for various special types of matrices such as totally positive matrices in Section~\ref{s:totpos} and inverse nonnegative matrices~\ref{s:invpos}. For another special class represented by a rank one radius matrix described in Section~\ref{s:rankone} the regularity radius computation in a general form, see equation~\eqref{eq:genregrad}, can be reduced to computation in simpler form, see equation~\eqref{eq:regrad}, as described in Section~\ref{s:rankone}. Moreover, this class inspired us also to design an approximation algorithm which is described in Section~\ref{s:rankoneapprox}. This closes the Section~\ref{s:specialmats} that concerns with special type of matrices.

The following sections assume that we have a general type of a matrix, or we are not aware how exactly our system's structure can be utilized. In these cases we need to apply steps considering general settings. One of the first questions that might be relevant in its application is checking finiteness of the regularity radius. Section~\ref{s:finiteness} provides arguments showing that testing unboundedness of the regularity radius is a polynomial problem. This might be helpful in situations, where we need to exclude extreme cases. Having finiteness decided one is usually willing to apply one of the simple bounds that is sufficient in considered application. There are several bounds known in the literature -- the commonly used bounds are reviewed above. Providing more freedom of choice, another bounds based on relationship between the maximum (Chebyshev) norm and the spectral norm is presented in Section~\ref{s:bounds}. Often, there are situations where no bounds are applicable and we need to compute exact or more accurate value of the regularity radius. As mentioned above, the common methods might involve application of various approximations algorithms. For various reasons, either numerical or processing ones, we might be forced to adopt different approaches, see discussion and argumentation along with definition of not a priori exponential algorithm by Kolev~\cite{Kol2011}. For these reasons we also suggest a new method that is not a priori exponential; see Section~\ref{s:ourmethod}.


\section{Special classes of matrices}\label{s:specialmats}

\subsection{Tridiagonal matrices}\label{ss:tridiagonalmat}

Considering the class of tridiagonal matrices, we can define the regularity radius as a perturbation of only non-zero elements of tridiagonal matrix. Following the notation from~\cite{Baron1996}, let~$T(b_k, a_k, c_k) \in \R^{n\times n}$ be a tridiagonal matrix defined as
\begin{equation*}
\left(\begin{matrix}
a_1 	& c_1 		& ~			& ~	\\
b_2 	& a_2 		& \ddots	& ~	\\
~ 		& \ddots 	& \ddots 	& c_{n-1}\\
~ 		& ~ 		& b_n 		& a_n
\end{matrix}\right)
\end{equation*}

We consider only unreduced tridiagonal matrices, i.e., those having $c_k, b_{k+1} \neq 0$ for all $1 \leq k \leq n - 1$. Let~$T$ be a real tridiagonal matrix $T(b_k, a_k, c_k)$, and let $\Delta \in \R^{n\times n}$ be a matrix having the same zero structure as~$T$. The regularity radius is defined the same way as in the introduction; this time only nonzero elements are perturbed. Bar-On, Codenotti and Leoncini~\cite{Baron1996} have shown that regularity of corresponding interval matrix can be determined in linear time. We can make use of this property to compute also the regularity radius.

\begin{proposition}\label{propTriDiagPol}
Let $A$ be a non-degenerate tridiagonal matrix and $\Delta$ having the same zero structure (i.e., $A_{ij}=0\ \Rightarrow\ \Delta_{ij}=0$). The regularity radius $\rr(A, \Delta)$ is computable in polynomial time with arbitrary precision.
\end{proposition}

\begin{proof}
Let $\rr_L$ and $\rr_U$ be lower and upper bounds on $\rr(A, \Delta)$. We can take $\rr_L=0$ and $\rr_U$ from \nref{rUpperBoundRohn}, for instance. Both of them have polynomial size. Now, we apply standard binary search on the interval given by these bounds.  In each of the steps, the actual approximation of the radius $\rr_i$ in fact determines tridiagonal interval matrix $\imace{A}_{\rr_i\Delta}$ for which we can use Bar-On, Codenotti and Leoncini algorithm~\cite{Baron1996} and compute its possible regularity in linear time. Considering the complexity of corresponding steps in binary search, we have the claim.
\end{proof}

This proposition provides a polynomial time algorithm to compute regularity radius. More precisely, the polynomial algorithm suggested in the above proof runs in time $\mathcal{O}\left(n(\rr_U-\rr_L)\log(1/\eps)\right)$. We can see that its time complexity depends on the size of the matrix entries. Therefore, it is still an open question whether a strongly polynomial algorithm exists. For a number $x$, define an operator $\sigma(x)$ as a size of its representation depending on chosen model of computation, e.g., the number of bits for Turing machine.

\begin{question} 
Let $A$ be a tridiagonal matrix. Let $\alpha$ be a rational number and $\sigma(\alpha)$ the size of its corresponding representation. Is the decision problem $r(A)\geq \alpha$ solvable in time $\mathcal{O}\left((n \cdot \sigma(\alpha))^k\right()$?
\end{question}

Notice that $\rr(A, \Delta)$ cannot be computed exactly in general because it can be an irrational number even for a rational input. Consider for example the matrix and the associated weights defined as
\begin{align*}
A=\begin{pmatrix}1&0\\0&1\end{pmatrix},\quad \text{and}\quad
\Delta=\begin{pmatrix}1&2\\1&1\end{pmatrix}.
\end{align*}
Then it is not hard to analytically determine that $\rr(A, \Delta)=\sqrt{2}-1$.

On the other hand, provided $\rr(A, \Delta)$ has polynomial size, the algorithm described in Proposition~\ref{propTriDiagPol} finds it exactly in polynomial time. Of course, this computational complexity is rather of theoretical nature. For real computation, the corresponding algorithm possibly providing an approximation that can be tuned with algorithm complexity should be designed. This leads us to another open question.

\begin{question} How to effectively design an approximate algorithm estimating the regularity radius for tridiagonal matrices based on Bar-On, Codenotti and Leoncini regularity testing~\cite{Baron1996} acquiring good approximation while maintaining feasible complexity?
\end{question}

\subsection{P-matrices and totally positive matrices}\label{s:totpos}

A matrix is a {\em P-matrix} if all its principal minors are positive. The task to check P-matrix property itself is an NP-hard problem~\cite{rum2003}. For some subclasses of P-matrices, computation of the regularity radius can be handled in an easier way.

One of the interesting subclasses of P-matrices from viewpoint of regularity are \emph{totally positive matrices}. A matrix $A\in\R^{n\times n}$ is totally positive if all its minors are positive. Despite the definition, checking this property is a polynomial problem; see Fallat and Johnson \cite{FalJoh2011}.

If $A$ is totally positive, then $\sgn(A^{-1}_{ij})=(-1)^{i+j}$. That is, the signs of the entries of the inverse matrix correspond to the checkerboard structure. Let us denote $s:=(1,-1,1,-1,\dots)^T$. Then we have $\sgn(A^{-1})=ss^T$, and therefore the regularity radius of $A$ reads
\begin{align*}
\rr(A)=\frac{1}{s^TA^{-1}s}.
\end{align*}

This is indeed a high simplification reducing original computation to a simple formula. Conditions determining classes of P-matrices are of great interest considering the regularity radius. We can think about another prominent subclass -- nonsingular M-matrices. This set of matrices is a subset of P-matrices as well as inverse nonnegative matrices. The latter class is again of a great interest, see below.

\subsection{Inverse nonnegative matrices}\label{s:invpos}

A matrix $A\in\R^{n\times n}$ is inverse nonnegative if $A^{-1}\geq0$. As mentioned above, inverse nonnegative matrices form a superclass of nonsingular M-matrices. For these we know that lower bounds~\eqref{eq:rumpbound} collapse to equality. We can however use inverse nonnegativity of matrices and compute the radius directly as:
\begin{align*}
\rr(A)=\frac{1}{e^TA^{-1}e}.
\end{align*}

The lower bound~\eqref{eq:rumpbound} is attained exactly as long as 
$$
\rho(|A^{-1}|\Delta)
=\max_{y,z\in\{\pm1\}^n}\rho_0(A^{-1}D_y\Delta D_z).
$$
This is the case if $D_zA^{-1}D_y\geq0$ for some $y,z\in\{\pm1\}^n$; see for example \cite{Roh2012a}. There are several classes of matrices satisfying this property. If $A$ is totally positive, then 
\begin{align*}
\rr(A,\Delta)
 = \frac{1}{\rho(|A^{-1}|\Delta)}
 = \frac{1}{\rho(D_sA^{-1}D_s\Delta)}.
\end{align*}
If $A$ is inverse nonnegative (e.g., an M-matrix), then 
\begin{align}\label{eq:invnonneg}
\rr(A,\Delta)
 = \frac{1}{\rho(|A^{-1}|\Delta)}
 = \frac{1}{\rho(A^{-1}\Delta)}.
\end{align}

This is again a high simplification motivating also further search within the class of P-matrices. Along with the mentioned simple cases this class contains also some hard instances. According to Poljak and Rohn~\cite{PolRoh1993, Roh1996} a problem to decide $d(A) \leq 1$ for a nonnegative symmetric positive definite relational matrix $A$ is NP-hard. This suggest that class of P-matrices exhibits some nice dichotomy from the viewpoint of regularity radius determination complexity. Considering these results we can ask the following question.

\begin{question} Find a subclass of P-matrices for which the determination of the regularity radius is the same as its determination for a general matrix.
\end{question}

\subsection{Rank one radius matrix}\label{s:rankone}

Suppose that $\Delta=uv^T>0$. Then the interval matrix $\imace{A}_{\delta\Delta}=[A-\delta uv^T,A+\delta uv^T]$ is regular if and only if the scaled interval matrix $[D_u^{-1}AD_v^{-1}-\delta ee^T,D_u^{-1}AD_v^{-1}+\delta ee^T]$ is regular. Hence we can reduce the computation of a general regularity radius, given by equation~\eqref{eq:genregrad}, to computation of the simpler one, given by equation~\eqref{eq:regrad}, as follows
\begin{align*}
\rr(A,uv^T) = \rr(D_u^{-1}AD_v^{-1}).
\end{align*}
For a similar results (without a proof) see Rohn~\cite{Roh2012b}.

\subsection{Approximation algorithm by rank one radius matrices}\label{s:rankoneapprox}

In this section we provide design of an approximate algorithm for regularity radius when $\Delta$ is a general radius matrix. Let $\Delta=U\Sigma V^T$ be SVD decomposition with singular values $\sigma_1,\dots,\sigma_r$. Define the rank one matrix $\Rad{B}:=U_{*1}\sigma_1 (V_{*1})^T$. By the construction of SVD we have  $\Rad{B}\geq0$. 
In some sense, $\Rad{B}$ is the best rank one approximation of $\Delta$, so $\rr(A,\Rad{B})$ approximates $\rr(A,\Delta)$. 

To evaluate quality of this approximation we can use the following procedure. Find maximal $\alpha$ and minimal $\beta$ such that $\alpha ee^T\leq\Delta\leq \beta ee^T$.
Then by simple application of some basic properties of regularity radius such as property in equation~\eqref{eq:rad_mult_constant} provides
\begin{align*}
\frac{1}{\beta} \rr(A,ee^T)
\leq \rr(A,\Delta)
\leq \frac{1}{\alpha} \rr(A,ee^T).
\end{align*}
So the quality of approximation is given by the ratio $\beta:\alpha$.

\begin{question} Evaluate quality of approximation and compare it to alternative solutions such as~\cite{Hartman2016}.
\end{question}

\begin{question} Is it possible to extend this approach using $uv^T\geq0$ instead of $ee^T$? What $u,v$ are the best?
\end{question}

\section{Finiteness of the radius}\label{s:finiteness}

For the regularity radius we have discussed several bounds as well as approximation algorithms. In many cases, especially for bounds, we have assumed that its value is finite. Unfortunately, its value is not necessarily finite as shown by following simple example.

\begin{example}\label{exInf}
Consider
\begin{align*}
A=\begin{pmatrix}1&1&1\\1&1&2\\1&2&3\end{pmatrix},\quad
\Delta=\begin{pmatrix}0&0&0\\0&0&0\\0&0&1\end{pmatrix}.
\end{align*}
Then any matrix $A'\in\imace{A}_{\delta\Delta}$ is nonsingular for every $\delta\geq0$ since $\det(A')=-1$ is constant. Therefore $\rr(A,\Delta)=\infty$.
\end{example}

This situation is of course a very special case. We can, however, check unboundedness of the radius in polynomial time, see the following theorem.

\begin{theorem}\label{thmFin}
Checking whether $\rr(A,\Delta)=\infty$ is a polynomial problem.
\end{theorem}

\begin{proof}
Without loss of generality assume $A$ is nonsingular, that is, $\rr(A,\Delta)\not=0$. Consider the interval matrix $\imace{A}_{\Delta}:=[A-\Delta,A+\Delta]$. It is regular iff the system
\begin{align}\label{sysOP}
|Ax|\leq\Delta|x|
\end{align}
has only trivial solution $x=0$.
Let $I$ be an index set of those $k$ for which $\Delta_{k*}=0$. For such a $k$ the $k$th inequality in the above system reads $|A_{k*}x|\leq0$, from which $A_{k*}x=0$. Denote by $\mna{P}$ the vector space defined by $A_{k*}x=0$, $k\in I$.

If there are $i,j$ such that $\Delta_{ij}>0$ and $x_j=0$ for every $x\in\mna{P}$, then put $\Delta_{ij}:=0$. If $\Delta_{i*}=0$, then put $I:=I\cup\{i\}$ and update $\mna{P}$ accordingly. Repeat this process until no change happens.

We claim that $\rr(A,\Delta)=\infty$ iff $I=\seznam{n}$.

\quo{$\Leftarrow$}
This is obvious since nonsingularity of $A$ implies $\mna{P}=\{0\}$, so \nref{sysOP} has only trivial solution.

\quo{$\Rightarrow$}
For every $i\not\in I$ choose $j_i$ such that $\Delta_{ij_i}>0$. We want to show solvability of the system
\begin{align}
\label{eqPfThmRinf1}
x_{j_i}&\not=0,\quad i\not\in I,\\
\label{eqPfThmRinf2}
A_{i*}x&=0,\quad i\in I
\end{align}
because it solution also solves \nref{sysOP} for a sufficiently large multiple of $\Delta$. First, we remove redundant constraints in \nref{eqPfThmRinf1} such that there are no $i,i'\not\in I$ satisfying $i\not=i'$ and $j_i=j_{i'}$. Denote the corresponding index set by $I'$. Thus, we want to find $y_i\not=0$, $i\in I'$, such that the system
\begin{align*}
x_{j_i}&=y_i,\quad i\in I',\\
A_{i*}x&=0,\quad i\in I
\end{align*}
is solvable. Infeasibility of this system occurs only if there are some linearly dependent rows in the constraint matrix. Since the first rows indexed by $I'$ are linearly independent, the only possibility is that some of the remaining rows are linearly dependent on the rows above them. Thus, to achieve solvability of the system, consider $y_i$s as unknowns, and the linear dependencies give rise to the system
$
By=0.
$
Notice that the solution set of this system is nontrivial and does not lie in any hyperplane of the form $y_i=0$, $i\in I'$: Otherwise $x_{j_i}=0$ is a consequence of \nref{eqPfThmRinf2}, which is a contradiction. In this case, however, the system $By=0$ has a solution $y$ with no zero entry, which completes the proof.
\end{proof}

In Example~\ref{exInf} we saw a matrix $A$ with infinite regularity radius and $\Delta$ having only one nonzero entry. So the natural question arises: How relates the number and positioning of nonzero entries of $\Delta$ with finiteness of $\rr(A,\Delta)$?

Obviously,  $\rr(A,\Delta)<\infty$ provided $\Delta>0$. Rohn~\cite[Thm.~83(ii)]{Roh2012b} proposed a stronger result (without a proof). For the sake of completeness, we present it here with a proof. 

\begin{proposition}[Rohn, 2012]
If $\Delta e>0$ or $\Delta^T e>0$, then $\rr(A,\Delta)<\infty$.
\end{proposition}

\begin{proof}
If $\Delta e>0$, then
\begin{align*}
|Ae|\leq\alpha\Delta e
\end{align*}
for a sufficiently large $\alpha>0$. Thus $x:=e$ is a solution of system \nref{sysOP}, meaning that there is a singular matrix in $\imace{A}_{\alpha\Delta}$.

Case $\Delta^T e>0$ is trivial by matrix transposition. 
\end{proof}

On the other hand, it may happen that $\rr(A,\Delta)=\infty$ even if $\Delta$ has many nonzero entries. Consider the example, where $A$ is the identity matrix of size $n$, and the radius matrix is defined as $\delta_{ij}=1$ if $i<j$ and $\delta_{ij}=0$ otherwise.
Then $\rr(A,\Delta)=\infty$ even though $\Delta$ has $\binom{n}{2}$ nonzero entries. We show that this is the maximum posible number of nonzero entries.

\begin{proposition} 
The maximal number of nonzero entries of $\Delta$ such that $\rr(A,\Delta)=\infty$ is $\binom{n}{2}$.
\end{proposition}

\begin{proof} 
Let $A\in\R^{n\times n}$ and $\Delta\geq0$ such that $\rr(A,\Delta)=\infty$. We have to show that the number of positive entries in $\Delta$ is at most $\binom{n}{2}$.
Consider system \nref{sysOP}. At the beginning of the procedure described in the proof of Theorem~\ref{thmFin}, we have at least one $k$ such that $\Delta_{k*}=0$, so that $|I|\geq1$. If $I\not=\seznam{n}$, there must be $i,j$ such that $\Delta_{ij}>0$ and $x_j=0$ for every $x\in\mna{P}$.
Since $\dim\mna{P}=n-|I|$, there can be at most $I$ such $j$s. This holds during the iterations.  However, if we already put some index $i'$ to $I$ in the previous iteration, the sub-space $\mna{P}$ is lying in some $x_{j'}=0$, and thus the number of novel indices $j$ is decreased accordingly. Therefore, the worst case is that at  each iteration, we add just one index $i$ to $I$, and $\Delta_{kj}>0$ for every  $k\not\in I$. The total number of positive entries of $\Delta$ then reads $0+1+2+\dots+(n-1)=\binom{n}{2}$.
\end{proof}

\section{Bounding the regularity radius using spectral norm}\label{s:bounds}

Due to equivalence of matrix norms \cite{HorJoh1985}, we have for the maximum (Chebyshev) and spectral norms of a matrix $A\in\R^{n\times n}$
\begin{align*}
\|A\|_{\max} \leq \|A\|_2 \leq n\|A\|_{\max}.
\end{align*}
Since the distance of $A$ to the nearest singular matrix in the spectral norm is equal to the minimum singular value $\sigma_{\min}(A)$, we obtain the bounds
\begin{align*}
\frac{1}{n}\sigma_{\min}(A) \leq \rr(A) \leq \sigma_{\min}(A).
\end{align*}
Even more, based on the properties of singular values, we propose another upper bound for $\rr(A)$, which is sometimes attained as the exact value. 


\begin{lemma}\label{lmmSingForm}
A nearest singular matrix to $A$ in the maximum norm has the form of $A-\rr(A)yz^T$ for some $y,z\in\{\pm1\}^n$.
\end{lemma}

\begin{proof}
Let $y,z\in\{\pm1\}^n$ be the maximizers of the formula \nref{eqRdel}. Then 
$
A^{-1}D_y\Delta D_z-\frac{1}{\rr(A)}I_n
$
is singular, so there is $x\not=0$ such that
$$
\left(A^{-1}D_y\Delta D_z-\frac{1}{\rr(A)}I_n\right) x =0.
$$
From this we derive
$$
\rr(A) D_y\Delta D_z x= A x,
$$
whence $A-\rr(A) D_y\Delta D_z$ is singular. Moreover, since $\Delta=ee^T$, we have $D_y\Delta D_z=D_yee^T D_z=yz^T$.
\end{proof}

Let $u$ and $v$ be the left and right singular vectors, respectively, corresponding to $\sigma_{\min}(A)$. Then $A-\sigma_{\min}(A)uv^T$ is a singular matrix, which is nearest to $A$ in the spectral norm. By Lemma~\ref{lmmSingForm}, a nearest matrix to $A$ in the maximum norm has the form of $A-\rr(A)yz^T$ for some $y,z\in\{\pm1\}^n$. Therefore, it is natural to set $y:=\sgn(u)$ and $z:=\sgn(v)$ and derive the upper bound based on \nref{eqRdel} and the proof of  Lemma~\ref{lmmSingForm}
\begin{align*}
\rr(A)
\leq \frac{1}{\rho_0(A^{-1}yz^T)}.
\end{align*}


\section{Not a priori exponential method}\label{s:ourmethod}

In this section, we describe a not a priori exponential method for computing the regularity radius $\rr(A,\Delta)$. It is based on the Jansson--Rohn \cite{JanRoh1999} algorithm for testing regularity of interval matrices. The advantage is that it is not a priori exponential, meaning that it may take exponential number of steps, but often it terminates earlier.

Put $b:=Ae$ and consider the interval linear system of equations $\imace{A}_{\delta}x=b$. The solution set is defined as
\begin{align*}
\{x\in\R^n\mmid \exists A\in\imace{A}_{\delta}: Ax=b\}.
\end{align*}
Since the solution set is non-empty, it is either a bounded connected polyhedron (in case of regular $\imace{A}_{\delta}$), or each topologically connected component of the solution set is unbounded (in case of singular $\imace{A}_{\delta}$); see Jansson \cite{Jan1997}. This is the underlying idea of the Jansson--Rohn algorithm for checking regularity of $\imace{A}_{\delta}$: Start within an orthant, containing some solution and check if the solution set in this orthant is bounded. If yes, explore recursively the neighboring orthants intersecting the solution set. Continue until you process the whole connected component.

Our situation is worse since we have to find the minimal $\delta\geq0$ such that $\imace{A}_{\delta}$ is singular. We suggest the following method. Notice that an orthant is characterized by a sign vector $s\in\{\pm1\}^n$ and described by $\diag(s)x\geq0$.
\begin{enumerate}
\item
Start in the non-negative orthant since $x=e$ is a solution. 
\item\label{step2}
Find minimal $\delta_e$ such that the corresponding solution set in this orthant is unbounded. 
\item\label{step3}
For each of the $n$ neighboring orthants find minimal $\delta_s$ such that the corresponding solution set intersects the orthant $\diag(s)x\geq0$.
\item
If $\delta_e\leq\delta_s$ for all neighboring orthants $s$, then we are done and $\rr(A,\Delta)=\delta_e$. Otherwise choose the minimal $\delta_s$, move to the neighboring orthant $s$ and repeat the process.
\end{enumerate}

Herein, the question is how to perform steps~\ref{step2} and ~\ref{step3}. 
According to the Oettli-Prager theorem \cite{OetPra1964}, the part of the solution set lying in the orthant $s\in\{\pm1\}^n$ is described by linear inequalities
\begin{align}\label{sysDs}
Ax-b \leq \delta \Delta \diag(s) x,\
-(Ax-b) \leq \delta \Delta \diag(s) x,\  
\diag(s) x\geq0.
\end{align}
The set is unbounded if and only if the recession cone has a non-trivial solution \cite{Schr1998}, that is, the system
\begin{align}\label{sysRecKzlDs}
Ax \leq \delta \Delta \diag(s) x,\ 
-Ax \leq \delta \Delta \diag(s) x,\ 
\diag(s) x\geq0,\ e^T\diag(s) x=1
\end{align}
is feasible. This leads us to the optimization problem
\begin{align*}
\min\ \delta \mbox{ subject to \nref{sysRecKzlDs}}
\end{align*}
to compute $\delta_e$ from step~\ref{step2}. This problem belongs to a class of generalized linear fractional programming problems, and it is solvable in polynomial time by using interior point methods \cite{FreJar1995, NesNem1995}.

Similarly we compute $\delta_s$ from step~\ref{step3} by solving the optimization problem
 \begin{align*}
\min\ \delta \mbox{ subject to \nref{sysDs}}.
\end{align*}

\section{Numerical experiments}\label{s:ex}

We have tested this approach numerically using the following settings. We have fixed various matrix sizes $N = \{3, 4, \ldots, 13\}$ and set number of instances for each size to be $M = 10$. For these settings we have generated collections of matrices $\mathcal{A}_{n,m}$, where $n \in N$ and $m \in {1, 2, \ldots, M}$. Each matrix $\mathcal{A}_{n, m} = \{a_{i,j}\}_{i,j = 1}^{n}$ is a random matrix of chosen type. For our purposes we have used three types of matrices:
\begin{enumerate}
\item {\em zero centered random matrix}, 
\item {\em random orthogonal matrix},
\item {\em (almost) inverse nonnegative matrix}.
\end{enumerate}

These matrices have been generated as follows. To generate {\em zero centered random matrix} we generated at first matrix $A$ of values generated uniformly from interval $[-5,5]$. The second matrix denoted as {\em random orthogonal} is simply generated using the previous process of generating zero centered matrices and consequent utilization of matrix Q from its QR decomposition.  The last type called here {\em (almost) inverse nonnegative} is parametric type of random matrices. Let us at first mention how to generate inverse nonnegative. These are generated as follows. We start with a matrix $B$ having its values generated uniformly from interval $[0,1]$. Then the inverse nonnegative matrix is generated simply by an inverse of matrix $B$, i.e., $A = B^{-1}$. This produces {\em inverse nonnegative matrix}. To acquire {\em almost inverse nonnegative} we have chosen randomly $d = 10$ percent of elements from inverse nonnegative matrix and negate each of them. The elements of a matrix $\Delta_{n,m}$ are generated randomly uniformly from interval $[0,1]$. 

All computations were performed in Matlab 2016b using function fminimax from Optimization Toolbox. Only one thread has been assigned to each computation to avoid effects of automatic parallelism. For each pair of generated matrices $\mathcal{A}_{n,m}$ and $\Delta_{n,m}$ we have estimated values of the regularity radius using either {\em the full search} analysis of equation~\eqref{eqRdel} producing thus a numerical estimation of the regularity radius denoted by $r_{FS}(\mathcal{A}_{n,m},\Delta_{n,m})$, or our suggested method of {\em orthant search} using the above described exploration of orthants producing thus an estimation $r_{OS}(\mathcal{A}_{n,m}, \Delta_{n,m})$. While the full search method might be expected as precise, we are interested in normalizing the difference of these two methods. For these reasons we define the following difference.
\begin{equation*}
	\delta_r = \frac{r_{OS} - r_{FS}}{r_{FS}}.
\end{equation*}

This could be computed for any values of $n$ and $m$. Since index $m$ is introduced for statistical reason, we are more interested in average value across this axis -- by $\overline{\delta}_r$ we denote an average over $m$.

\begin{figure}[htbp]
  \centering
  \includegraphics[width=14cm]{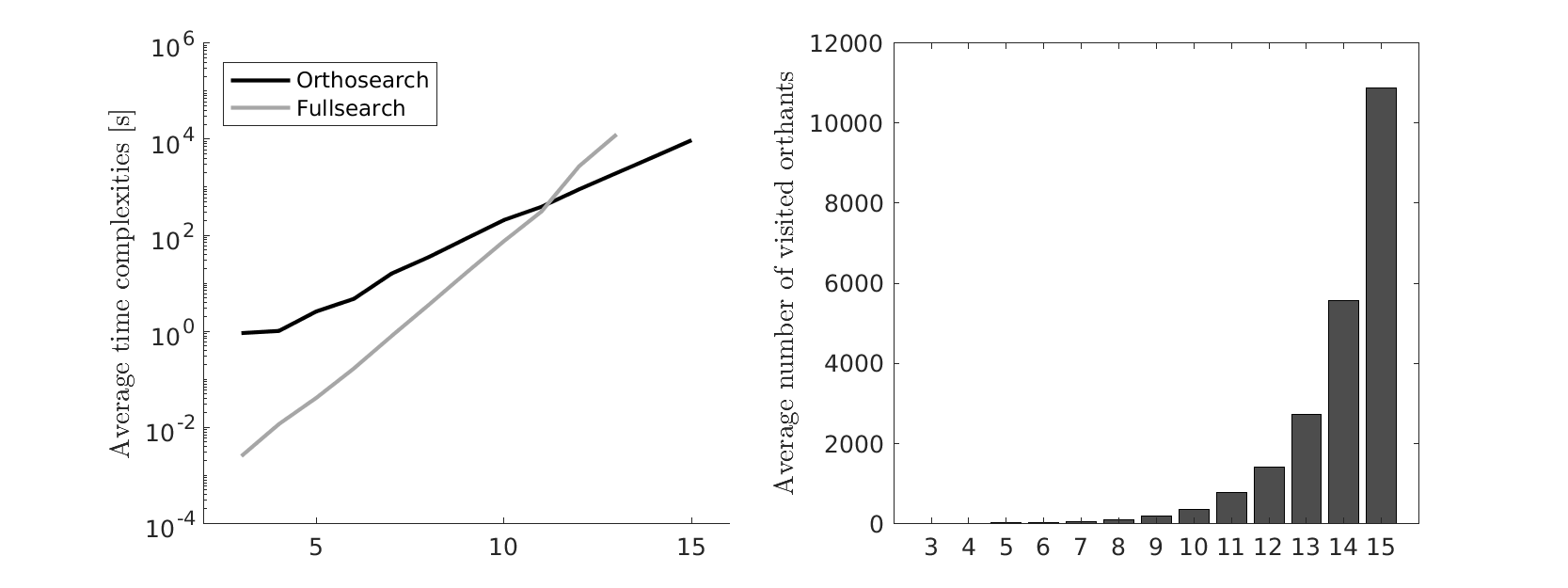}
  \caption{Results of testing estimation of the regularity radius using the full search and orthant search methods for {\em random zero centered matrices}. Presented are time complexities (left) and number of visited orthants for orthant search method (right).}
  \label{fig:zero_centered}
\end{figure}

As mentioned above, we have also recorded time consumption. Although the corresponding conditions were preset in order to maintain comparable situations, we have also decided to include an auxiliary characteristic enabling to evaluate time complexity. This characteristic is represented by the number of visited orthants for which the optimization of problem~\eqref{sysRecKzlDs} took place. For the orthant search method we consider matrices of extended sizes $N' = \{3, 4, \ldots, 15\}$ to better show the growing number of visiting orthants. This is possible due to reasonable time complexities for larger matrices.

\subsubsection*{Zero centered matrices}
At first, let us explore results of computation for {\em random zero centered matrices} as shown in Figure~\ref{fig:zero_centered}. For the full interval of matrix sizes $N$ average differences between radii $\overline{\delta}_r$  are numerically negligible.
The results suggest that our orthant search provides the same estimation compared to the full search method. At the same time its complexity is growing slower.

\subsubsection*{Orthogonal matrices}
Figure~\ref{fig:rand_orthog} presents results of second group of matrices, namely {\em random orthogonal matrices}. These have shown roughly the same behavior as the above mentioned results for zero centered matrices. This holds also for the number of visited orthants.

\begin{figure}[htbp]
  \centering
  \includegraphics[width=14cm]{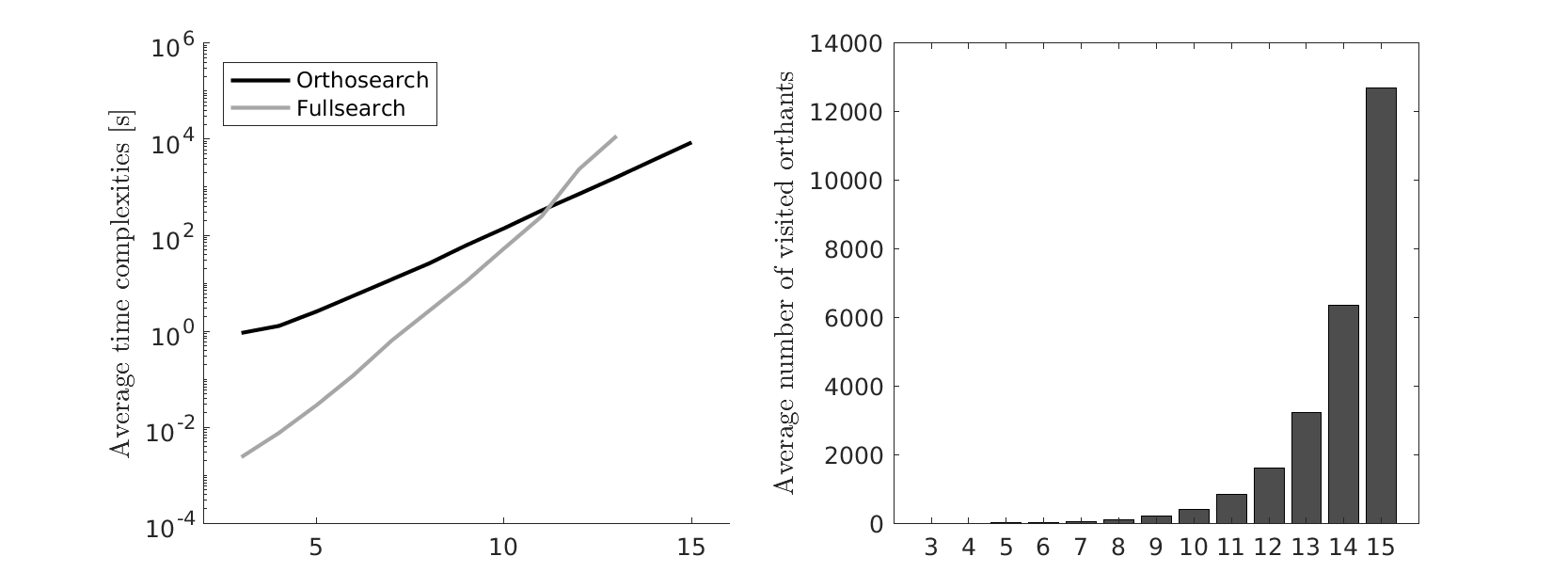}
  \caption{Results of testing estimation of the regularity radius using the full search and orthant search methods for {\em random orthogonal matrices}. Presented are differences $\overline{\delta}_r$ (left), time complexities (middle) and number of visited orthants for orthant search method (right).}
  \label{fig:rand_orthog}
\end{figure}

\subsubsection*{(Almost) inverse non-negative matrices}
We have also tested this approach for inverse non-negative matrices. 

\begin{figure}[htbp]
  \centering
  \includegraphics[width=14cm]{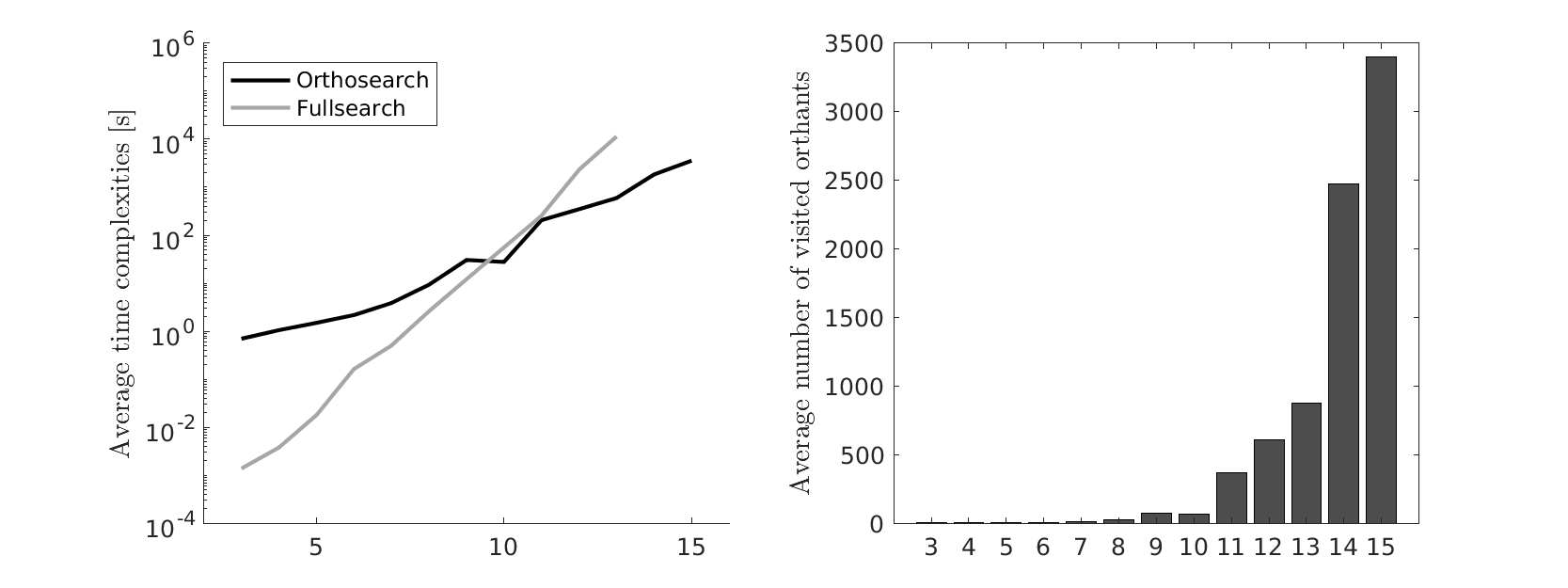}
  \caption{Results of testing estimation of the regularity radius using full search and orthant search methods for {\em random almost inverse nonnegative matrices}. Presented are differences $\overline{\delta}_r$ (left), time complexities (middle) and number of visited orthants for orthant search method (right).}
  \label{fig:almost_inv_noneg}
\end{figure}

For this class we have already derived an exact formula shown in equation~\eqref{eq:invnonneg} and computation of corresponding radius in this way is meaningless. The reason for testing this class may make sense for testing of methods validity. As expected results of testing these matrices have shown that in this case only one orthant is visited within each computation and thus these computations are very easy and their complexity is low. This, however, suggest that there exist matrices where the number of visited orthants is relatively low. To obtain results for less trivial class we have utilize class of {\em almost inverse nonnegative matrices}. Results of testing these matrices are shown in Figure~\ref{fig:almost_inv_noneg}.

We can see that most of the properties of the orthant search method is again exhibited. Moreover, due to relative proximity to structure of simple inverse nonnegative matrices, the number of visited orthants is significantly lower compared to previous cases. This indicates that having a nice structure of input matrix, although it might not be theoretically tractable, might represent a significant improvement of time complexity when regularity radius is handled via the orthant search method suggested above. Moreover, due to search character of the method its computation seems to be simply parallelized. 

\section{Conclusion}

While the problem of computing the regularity radius is NP-hard in general, a common approach to estimate its values might include an adoption of theoretical bounds or utilization of approximation algorithms. This usually means that we try to make use of a specific structure of the corresponding matrix to provide better results or to apply an approximation algorithm if necessary. This work presents results for regularity radius along such process of its analysis. The first series of results is concerned with possible special types of matrices. For a tridiagonal matrix we employ the published algorithm for regularity testing to construct an algorithm to determine polynomially the corresponding regularity radius. The problem is that this complexity depends on values of matrix elements and thus it remains open whether there exists a strongly polynomial algorithm. We also provide several exact formulas for regularity radius when the matrix is of a specific type -- namely for totally positive matrices and inverse nonnegative matrices (including M-matrices). Moreover, for rank one radius matrices we are able to transform the computation of the general radius of regularity, see equation~\eqref{eq:genregrad}, to its simpler form, see equation~\eqref{eq:regrad}. This enables us to use the approximation algorithm by Hartman and Hlad\'{i}k~\cite{Hartman2016} to acquire any predefined accuracy. Based on last mentioned result we also provide a construction of a new approximation algorithm using rank one approximation of a general radius matrix. Moving to general matrices, we start with a proof that checking whether regularity radius is infinite is a polynomial problem. Moreover we provide sharp bounds on the number of nonzero entries of radius matrix to enable infinite regularity radius. For a general matrix we also provide new bounds based on relationship between the maximum (Chebyshev) norm and the spectral norm. Finally, a new not a priori exponential approximation algorithm based on iterative search through orthants using Oettli-Prager theorem is presented along with numerical results. The presented results suggest that the method has typically significantly lower complexity compared to the full search computation based on equation~\eqref{eqRdel}. Moreover for some special types of matrices, here represented as almost inverse nonnegative, the number of visited orthants can be highly reduced.



\bibliography{rad_reg_ELA}

\end{document}